\documentclass[11pt]{amsart}
\usepackage{}
\usepackage[T1]{fontenc}
\usepackage{lmodern}
\usepackage{cite}
\usepackage{ifthen}
\usepackage{amsfonts}
\usepackage{amsxtra}
\usepackage{amssymb}
\usepackage{amsthm}
\usepackage{array}
\usepackage[margin=1in]{geometry}
\usepackage{xcolor}
\definecolor{cite}{rgb}{0.50,0.00,1.00}
\definecolor{url}{rgb}{0.00,0.50,0.75}
\definecolor{link}{rgb}{0.00,0.00,0.50}
\usepackage[colorlinks,linkcolor=link,urlcolor=url,citecolor=cite,pagebackref,breaklinks]{hyperref}
\usepackage{mathtools}
\usepackage{mathrsfs}
\usepackage[all]{xy}
\usepackage[lite,abbrev,msc-links]{amsrefs}

\makeindex

\theoremstyle{definition} 
\newtheorem{Unity}{Unity}[section] 
\newtheorem*{defn*}{Definition} 
\newtheorem{defn}[Unity]{Definition} 

\theoremstyle{plain} 
\newtheorem*{thm*}{Theorem}
\newtheorem{thm}[Unity]{Theorem}

\newtheorem*{cor*}{Corollary}
\newtheorem{cor}[Unity]{Corollary}
\newtheorem{lem}[Unity]{Lemma}

\theoremstyle{remark} 
\newtheorem*{rmk*}{Remark}
\newtheorem{rmk}[Unity]{Remark}

\numberwithin{Unity}{section}
\newcommand{\diff}{{\rm d}}

\newcommand{\surj}{\twoheadrightarrow}

\newcommand{\Char}{{\rm char\ }}
\newcommand{\Spec}{{\rm Spec \,}}
\newcommand{\rank}{{\rm rank\ }}
\newcommand{\Pic}{{\rm Pic}}

\newcommand{\sE}{\mathscr{E}}
\newcommand{\sF}{\mathscr{F}}

\newcommand{\sL}{\mathscr{L}}

\begin{document}
\title{On the stability of tangent bundles on cyclic coverings}
\author[Yongming Zhang]{Yongming Zhang}
\email{zhangyongming@amss.ac.cn}
\address{School of Mathematical Sciences, Fudan University, Shanghai, 200433,
	China}
\maketitle
\begin{abstract}
Let $Y$ be a smooth projective surface defined over an algebraically closed field $k$ with  $\Char k\nmid n$, and let $\pi:X\rightarrow Y$ be a $n$-cyclic covering  branched along a smooth divisor $B$. We show that under some conditions $\mathscr{T}_X$ is semi-stable with respect to $\pi^*\mathcal {H}$ if the tangent bundle $\mathscr{T}_Y$ is semi-stable with respect to some ample line bundle $\mathcal {H}$ on $Y$.
\end{abstract}
\section{Introduction}

In this paper we correct some results in \cite{Zhang} and improve them to cyclic coverings.
Let $\pi:X\rightarrow Y$ be a finite covering defined over $k$. We know that the pull back of semi-stable sheaves under $\pi^*$ are semi-stable in characteristic zero(see \cite[lemma 1.1]{GD}), but it does not generally hold that the pull back of stable sheaves are stable. In some special cases, the pull back of stable sheaves are stable:
\begin{thm}(Theorem \ref{thm3.2})
Let $Y$ be a smooth projective variety over $k$  with $\Char k\nmid n $ and let $\pi: X\rightarrow Y$ be a $n$-cyclic covering branched along a smooth effective divisor $B\in \Gamma(Y,\mathcal {L}^{\otimes n})$ for some line bundle $\mathcal {L}\in \Pic(Y)$. If $\mathcal {F}$ is locally free sheaf on $Y$ and (semi-)stable with respect to an ample line bundle $\mathcal {H}$ on $Y$ then $\pi^*\mathcal {F}$ is (semi-)stable with respect to the ample line bundle $\pi^*\mathcal {H}$.
\end{thm}

Then we study the stability of tangent bundle on a surface which is a $n$-cyclic covering of another surface and obtain the following results.
\begin{thm}(Theorem \ref{thm3.5})\label{thm1.1}
Let $\pi: X\rightarrow Y$ be a $n$-cyclic covering determined by a line bundle $L \in \Pic (Y)$ and a smooth divisor $B \in \Gamma(Y, L^{\otimes n})$, where $Y$ is a smooth projective surface over an algebraically closed field $k$ with $\Char k\nmid n$ . And let $\mathcal {H}\in \Pic(Y)$ be an ample line bundles satisfying $\mathcal {H} \equiv l B$ for some $0<l\in \mathbb{Q}$.
Assume that the  inequality holds:
$n\deg \Omega_{Y}+(n+1) \deg B\ \geq0$ and the cotangent bundle $\Omega_{Y}$ is semi-stable with respect to $\mathcal {H}$. Then $\Omega_{X}$ is semi-stable with respect to $\pi^*\mathcal {H}\in \Pic(X).$
\end{thm}
\begin{thm}(Theorem \ref{thm3.6})
With the same conditions in Theorem \ref{thm1.1}, if the strict inequality holds: $n\deg \Omega_{Y}+(n+1) \deg B\ >0$ and the cotangent bundle $\Omega_{Y}$ is semi-stable with respect to $\mathcal {H}$, then $\Omega_{X}$ is stable with respect to $\pi^*\mathcal {H}\in \Pic(X)$.
\end{thm}

Using this result we find a class of surfaces whose tangent bundle are stable:
\begin{cor}(Corollary \ref{cor3.8})
Let $k$ be an algebraically closed field with  $\Char k\nmid n$ and $\pi: X\rightarrow \mathbb{P}^2_k$ be a $n$-cyclic covering branched along a smooth curve $B$ in $\mathbb{P}^2_k$. Then $\Omega_X$ is semi-stable with respect to $\pi^*\mathcal {O}_{\mathbb{P}^N}(1)$, and it is stable with respect to $\pi^*\mathcal {O}_{\mathbb{P}^N}(1)$ when $n=2$ and $\deg B\geq 4$ or $n>2$.

\end{cor}

In addition, we already know that the tangent bundle of a K3 surface of characteristic zero is stable (see Chapter 9 of \cite{HDLM}). By the above corollary we obtain a partial result about the stability of the tangent bundle of a K3 surface in positive characteristic: when   $\Char k\neq2$, the tangent bundle of a K3 surface,  which can be realized as a double covering of $\mathbb{P}^2$ branched along a smooth curve of degree $6$ or a $4$-cyclic covering of $\mathbb{P}^2$ branched along a smooth curve of degree $1$, is stable.

\section{Preliminaries}
\subsection{cyclic coverings}
Let $k$ denote an algebraically closed field with $\Char k\nmid n$. First we recall the definition of $n$-cyclic coverings. Let $Y$ be a smooth variety over $k$ and let $B$ be an effective divisor on $Y$. Suppose that $\mathcal {O}_Y(B)$ is divisible by $n$ in $\Pic (Y)$,
i.e. there is a line bundle $\sL\in \Pic (Y)$ such that $\sL^{\otimes n}=\mathcal {O}_Y(B).$
Note that the divisor $B$ gives a morphism $\sL^{\otimes -n}\rightarrow \mathcal {O}_Y$, and this
morphism gives $\mathscr{E}=\mathcal {O}_Y\oplus \sL^{-1}\oplus\cdots\oplus\sL^{-n+1}$ a ring structure.
Let $X=\Spec(\mathscr{E})$, then with the canonical projection we get a $n$-cyclic covering over $Y$: $$\pi: X\longrightarrow Y.$$
 Let $B_1$ denote the reduced divisor of the inverse image $\pi^{-1}(B)$.
 Then we have the following lemma:
\begin{lem}\label{lem2.1}(\cite[1.17]{BPV})
Keep the above notation. Then we have:
\begin{enumerate}
\item $\pi^{\ast}\sL=\mathcal {O}_X(B_1)$;
\item \label{it3}$K_X=\pi^{\ast}(K_Y\otimes \sL^{\otimes(n-1)})$;
\item \label{it4}$\pi_{\ast}\mathcal {O}_X= \mathcal {O}_Y\oplus \sL^{-1}\oplus \cdots\oplus \sL^{\otimes(1-n)}.$
\end{enumerate}
\end{lem}

\subsection{Stability}
We collect here some well known facts about stability of coherent sheaves. Let $X$ be a smooth projective variety of dimension $n$ over an algebraically closed field $k$ and $\mathcal {H}$ an ample line bundle on $X$. Let $E$ be a torsion  free sheaf on $X$. The rank of $E$, denoted by $\rank E$, is the rank of $E$ at the generic point of $X$. The first Chern class of $E$ is denoted by $c_1(E)$, then the degree of $E$, denoted by $\deg E$, is the intersection number $c_1(E)\cdot\mathcal {H}^{n-1}$. We define:
$$\mu(E):=\frac{\deg E}{\rank E}$$
which depends on the choice of the ample line bundle $\mathcal {H}$ on $X$.

\begin{defn}
$E$ is stable (resp. semi-stable) if for any sub-sheaf $F$ with $0<\rank F<\rank E$, we have:
$$\mu(F)<\mu(E) \ \ (\text{resp.}\ \mu(F)\leq\mu(E)).$$
\end{defn}

Actually, in the above definition we only need to consider the saturated sub-sheaves of $E$.
If $E$ is a torsion free sheaf on $X$, then there is a unique filtration(the Harder-Narasimhan filtration)$$0=E_0\subsetneq E_1\subsetneq...\subsetneq E_l=E$$
such that the factors $gr_i^{HN}:=E_i/E_{i-1}$ for $i=1,...,l$ are semi-stable sheaves and $$\mu_{max}(E):=\mu(gr_1^{HN})>...>\mu(gr_l^{HN})=:\mu_{min}(E).$$
The instability of $E$ is defined as ${\rm I}(E):=\mu_{max}(E)-\mu_{min}(E)$, and $E$ is semi-stable if and only if ${\rm I}(E)=0$.

If $E$ is a semi-stable torsion free sheaf on $X$, then there is a filtration (the Jordan-H\"{o}lder  filtration)$$0=E_0\subsetneq E_1\subsetneq...\subsetneq E_l=E$$
such that the factors $gr_i=E_i/E_{i-1}$ for $i=1,...,l$ are stable sheaves and $$\mu(gr_1)=...=\mu(gr_l)=\mu(E).$$
Note that the Jordan-H\"{o}lder  filtration is not unique but the factors $gr_i:=E_i/E_{i-1}$ for $i=1,...,l$ are unique determined by $E$.

\section{The main results and Proof}
\subsection{stability of the pull back of stable sheaves}
For a Galois covering $\pi: X\rightarrow Y$ with Galois group $G$ and coherent sheaf $\sE$ on $Y$, then there is an isomorphism $\sigma^*\pi^*\sE\simeq \pi^*\sE$.  For any sub-sheaf $\sF\subset \pi^*\sE$ we will denote the image of sub-sheaf $\sigma^*\sF\subset \sigma^*\pi^*\sE$ under this isomorphism by $\sF^{\sigma}$ for $\sigma\in G$. \\
Next we relate the following fact in Galois decent theory:
\begin{thm}\label{thm3.1}(\cite[Theorem 28, Chapte 9]{FR})
Let $p:X\rightarrow Y$ be a finite flat Galois morphism with Galois group $G$, and suppose that $X$ and $Y$ are both integral. Let $\mathcal {F}$ be a locally free sheaf on $Y$ and suppose that $\mathcal {G}^\prime$ is a coherent sub-sheaf of $p^*\mathcal {F}$ such that:
\begin{itemize}
  \item For all $\sigma\in G$, $(\mathcal {G}^\prime)^\sigma=\mathcal {G}^\prime$ as a sub-sheaf of $p^*\mathcal {F}$;
  \item $p^*\mathcal {F}/\mathcal {G}^\prime$ is torsion free.
\end{itemize}
Then there exists a coherent sub-sheaf $\mathcal {G}$ of $\mathcal {F}$ such that $\mathcal {G}^\prime=p^*\mathcal {G}$ as a sub-sheaf of $p^*\mathcal {F}$
\end{thm}
Immediately, we have the following observation:
\begin{thm}\label{thm3.2}
Let Y be a smooth projective variety over $k$ with $\Char k\nmid n$ and let $\pi: X\rightarrow Y$ be a $n$-cyclic covering  branched along a smooth effective divisor $B\in \Gamma(Y,\mathcal {L}^{\otimes n})$ for some line bundle $\mathcal {L}\in \Pic(Y)$. If $\mathcal {F}$ is locally free sheaf on $Y$ and (semi-)stable with respect to an ample line bundle $\mathcal {H}$ on $Y$ then $\pi^*\mathcal {F}$ is (semi-)stable with respect to the ample line bundle $\pi^*\mathcal {H}$.
\end{thm}
\begin{proof}
Let $G=\{1,\sigma,\cdots,\sigma^{n-1}\}$ be the Galois group acting on $X$.
We first prove $\pi^*\mathcal {F}$ is semi-stable. Suppose $\pi^*\mathcal {F}$ is not semi-stable. then there is a unique maximal destabilizing sub-sheaf $$\mathcal {M}\subset\pi^*\mathcal {F}\ \text{ with }\ \mu(\mathcal {M})>\mu(\pi^*\mathcal {F}).$$ Note that $\mathcal {M}$ is a saturated sub-sheaf and $\mathcal {M}^\sigma=\mathcal {M}$. By Theorem \ref{thm3.1}, $\mathcal {M}=\pi^*\mathcal {G}$ for some sub-sheaf $\mathcal {G}\subset\mathcal {F}$. Since $\mu(\pi^*\mathcal {G})=\text{deg}(\pi)\ \mu(\mathcal {G})$ and $\mu(\pi^*\mathcal {F})=\text{deg}(\pi)\ \mu(\mathcal {F})$ we get $\mu(\mathcal {G})>\mu(\mathcal {F})$ which is a contradiction.

 Second,  we suppose $\pi^*\mathcal {F}$ is not stable. Then there is a filtration (the Jordan-H\"{o}lder  filtration)$$0=\mathcal {F}_0\subsetneq \mathcal {F}_1\subsetneq...\subsetneq \mathcal {F}_l=\pi^*\mathcal {F}$$
such that the factors $gr_i=\mathcal {F}_i/\mathcal {F}_{i-1}$ for $i=1,...,l\ (l\geq2)$ are stable sheaves and $$\mu(gr_1)=...=\mu(gr_l)=\mu(\pi^*\mathcal {F}).$$
There is no saturated sub-sheaf $\mathcal {E}\subset\pi^*\mathcal {F}$ with $\mu(\mathcal {E})\geq\mu(\pi^*\mathcal {F})$ and $\mathcal {E}^\sigma=\mathcal {E}$ (if not then $\mathcal {E}=\pi^*\mathcal {E}^\prime$ for some sub-sheaf $\mathcal {E}^\prime\subset\mathcal {F}$ and by the argument as before we will get a contradiction), so $\mathcal {F}_i^\sigma\neq\mathcal {F}_i$ for $i<l$. Using the fact that the only nontrivial map between two stable sheaves with the same slope is an isomorphism, we have $l\mid n$, $$\pi^*\mathcal {F}=gr_1\oplus \cdots\oplus gr_l\ \text{ and }\ gr_1\simeq\cdots \simeq gr_l.$$
More over $G$ acts on those factors transversally.
In fact, after the action of $\sigma$ we get another Jordan-H\"{o}lder  filtration:
$$0=\mathcal {F}_0^\sigma\subsetneq \mathcal {F}_1^\sigma\subsetneq...\subsetneq \mathcal {F}_l^\sigma=\pi^*\mathcal {F}$$
such that the factors $gr_i^\prime=\mathcal {F}_i^\sigma/\mathcal {F}_{i-1}^\sigma$ for $i=1,...,l$ are stable sheaves and $$\mu(gr_1^\prime)=...=\mu(gr_l^\prime)=\mu(\pi^*\mathcal {F}).$$
Let $l_0$ be the smallest number satisfies $\mathcal {F}_1^\sigma\subset\mathcal {F}_{l_0}$, then we get a nontrivial morphism between stable sheaves with the same slope: $\mathcal {F}_1^\sigma\rightarrow gr_{l_0}=\mathcal {F}_{l_0}/\mathcal {F}_{l_0-1},$
so it is an isomorphism, and $\mathcal {F}_1^\sigma\cap\mathcal {F}_1=0$. It follows that:
$\mathcal {F}_1^\sigma\oplus\mathcal {F}_1\subseteq\pi^*\mathcal {F}$. If we compare the original filtration with the filtration obtained by the action of $\sigma^2$ it we will get $\mathcal {F}_1^{\sigma^2}\oplus\mathcal {F}_1^\sigma\oplus\mathcal {F}_1\subseteq\pi^*\mathcal {F}$. Let us do this process and we will get the result finally.

On the other hand, after pushing forward we obtain:
$$\pi_*\pi^*(\mathcal {F})=\mathcal {F}\oplus\mathcal {F}(\mathcal {L}^{-1})\oplus\cdots\oplus\mathcal {F}(\mathcal {L}^{-n+1})=\pi_*gr_1\oplus \pi_*gr_2\oplus\cdots\oplus \pi_*gr_l.$$
Note that  $\sigma$ acts on $\mathcal {F}(\mathcal {L}^{-i})$ by multiplying $\xi^{i}$ where $\xi$ is the $n$-th root of unity, and on $\pi_*gr_1,\; i=1,\cdots,l$ transversally. Consider the map $\pi_*gr_1\rightarrow\mathcal {F}(\mathcal {L}^{-i}),a\mapsto\sum_{j=0}^{l-1} (\sigma^j)^*\xi^{ij}a$ at the generic point we will find that $\rank \pi_*gr_1=\rank \mathcal {F}$ and $\pi_*gr_1\cap \mathcal {F}\oplus\mathcal {F}(\mathcal {L}^{-1})\oplus\cdots\oplus\mathcal {F}(\mathcal {L}^{-n+2})=0$,
hence we can get an injection:
$$\pi_*gr_1\hookrightarrow\mathcal {F}(\mathcal {L}^{-n+1}),$$
whose quotient is a torsion sheaf. So we will get
$$\mu(\pi_*gr_1)\leq\mu(\mathcal {F}(\mathcal {L}^{-n+1}))<\mu(\mathcal {F}\oplus\mathcal {F}(\mathcal {L}^{-1})\oplus\cdots\oplus\mathcal {F}(\mathcal {L}^{-n+1}))=\mu(\pi_*\pi^*(\mathcal {F})),$$
which contradicts to $\mu(\pi_*\pi^*(\mathcal {F}))=\mu(\pi_*gr_1)$. So $\pi^*\mathcal {F}$ is stable.
\end{proof}

\subsection{sub-sheaves in cotangent bundle}
Let $\pi: X\rightarrow Y$ be a $n$-cyclic covering determined by a line bundle $L \in \Pic (Y)$ and a smooth divisor $B \in \Gamma(Y, L^{\otimes n})$, where $Y$ is a smooth projective variety over an algebraically closed field $k$ with $\Char k\nmid n$ . Let $B_1$ denote the reduced divisor of the inverse image $\pi^{-1}(B)$.

Now we consider a saturated sub-sheaf $\mathcal {M}_1\subset\Omega_X$ which is stable under action of $\sigma$.

Consider the following diagram
$$\xymatrix{
0\ar[r] &{\mathcal {N}_2} \ar[r] &{\mathcal{M}_2}\ar[r] &{Q_2}\ar[r] &{0}
\\
0\ar[r] &{\pi^*\Omega_{Y}} \ar[r] \ar@{->>}[u] &{\Omega_X}\ar[r]\ar@{->>}[u] &{\Omega_{X/Y}}\ar[r]\ar@{->>}[u] &{0} \\
0\ar[r] &{\mathcal {N}_1} \ar[r] \ar@{^{(}->}[u] &{\mathcal{M}_1}\ar[r]^{\varphi}\ar@{^{(}->}[u] &{Q_1}\ar[r]\ar@{^{(}->}[u] &{0}
}\leqno(3.1)\label{diag3.1}
$$
 where $\mathcal {N}_1=\pi^*\Omega_{Y}\cap\mathcal {M}_1\subset \pi^*\Omega_{Y}$ is also saturated and stable under the action of $\sigma$. By lemma \ref{thm3.1} there is a sub-sheaf $N_1^{\prime}\subset\Omega_{Y}$ such that $$\pi^*N_1^{\prime}=N_1.$$

 Let $(x_1,x_2,\cdots,x_m)$ be local coordinates on $Y$, centred at a point $x\in Y$, such that $B$ is defined by $x_1=0$. Then $(t,\pi^*x_2,\cdots,\pi^*x_m)$ is local coordinates on $X$, centred at $\pi^{-1}(x)\in X$, where $t$ is the vertical coordinate in $L$ satisfying $\pi^*x_1-t^n=0$. Locally any element of $\mathcal {M}_1$ should looks like $$t^{n-1}\sum_{i\geq2} \pi^*b_{n-1, i}\diff\pi^*x_i+\cdots+t\sum_{i\geq2} \pi^*b_{1,i}\diff\pi^*x_i+\pi^*c\diff t,\, \text{where}\, b_{j,i}, c\in \mathcal {O}_Y,$$
 where $\sigma$ acts on it by multiplying $n$-th unit root $\xi$ on $t$ and $\diff t$.

 Since $\mathcal {M}_1$ is stable under action of $\sigma$, locally it can always be generated by $1$-forms of the two kinds $$\sum_{i\geq2} \pi^*a_i\diff\pi^*x_i,\, \text{where}\, a_i\in \mathcal {O}_Y$$ and $$t\sum_{i\geq2} \pi^*b_i\diff\pi^*x_i+\pi^*c\diff t,\, \text{where}\, b_i, 0\neq c\in \mathcal {O}_Y.$$
We will call them the first and second kind generators respectively.
 It is easy to find that $\mathcal {N}_1=\mathcal {M}_1$ if and only if $\mathcal {M}_1$ can be generated by the first kind generators only.

 If $\mathcal {M}_1\neq\mathcal {N}_1$ then its generators should contain the second kind. 
And by some local calculation  we find that $Q_1$ is a sub-sheaf of $\Omega_{X/Y}$ of $\rank \,n-1$ as
$\mathcal {O}_{B_1}$-module. In this case we consider the following diagram
$$\xymatrix{
0\ar[r] &{I_B/I^2_B} \ar[r] &{\Omega_{Y}|_B}\ar[r] &{\Omega_{B}}\ar[r] &{0} \\
0\ar[r] &{I_B/I^2_B\cap\mathcal {N}_1^{\prime}|_B} \ar[r]\ar@{^{(}->}[u] &{\mathcal{N}_1^{\prime}|_B}\ar[r]\ar@{^{(}->}[u] &{\mathcal{G}}\ar[r]\ar@{^{(}->}[u] &{0}
.}\leqno(3.2)
$$
First, Multiplying a second kind generators of $\mathcal {M}_1$ by $t^{n-1}$, $$t^{n-1}[t\sum_{i\geq2} \pi^*b_i\diff\pi^*x_i+\pi^*c\diff t]=\pi^*x_1\sum_{i\geq2} \pi^*b_i\diff\pi^*x_i+\frac{1}{n}\pi^*c\diff \pi^*x_1,$$
we will get a element in $I_B/I^2_B\cap\mathcal {N}_1^{\prime}|_B$ under the restriction $\Omega_{Y}\twoheadrightarrow \Omega_{Y}|_{B}$. Conversely, for any element in $I_B/I^2_B\cap\mathcal {N}_1^{\prime}|_B$, we could get an element in $\mathcal {N}_1 $ which can factorize as a second kind generator and $t^{n-1}$. Since $\mathcal {M}_1$ is saturated, $\mathcal {M}_1$ would have the second kind generators.
So we get the following observation.
\begin{lem}\label{subsheaves}
$\mathcal {M}_1$ is not contained in $\pi^*\Omega_{Y}$ if and only if $$\mathcal {N}_1^{\prime}|_{B}\cap I_B/I^2_B\neq 0.$$
And in this case, $\rank Q_1=n-1$ as $\mathcal {O}_{B_1}$-module.
\end{lem}

\subsection{stability of tangent bundle on cyclic coverings}
\begin{thm}\label{thm3.5}
Let $\pi: X\rightarrow Y$ be a $n$-cyclic covering determined by a line bundle $L \in \Pic (Y)$ and a smooth divisor $B \in \Gamma(Y, L^{\otimes n})$, where $Y$ is a smooth projective surface over an algebraically closed field $k$ with char $\Char k\nmid n$ . And let $\mathcal {H}\in \Pic(Y)$ be an ample line bundles satisfying $\mathcal {H} \equiv l B$ for some $0<l\in \mathbb{Q}$.
Assume that the  inequality holds:
$n\deg \Omega_{Y}+(n+1) \deg B\ \geq0$ and the cotangent bundle $\Omega_{Y}$ is semi-stable with respect to $\mathcal {H}$. Then $\Omega_{X}$ is semi-stable with respect to $\pi^*\mathcal {H}\in \Pic(X).$

\end{thm}
\begin{proof}
%

Suppose that $\Omega_{X}$ is not semi-stable, then there is a unique maximally destabilizing sub-line bundle $\mathcal {M}_1\subset \Omega_{X}$ which is saturated and stable under the action of $\sigma$.

If $\mathcal {M}_1\subset\pi^*\Omega_{Y}\subset\Omega_{X}$,  then by lemma \ref{thm3.2} $\pi^*\Omega_{Y}$ is semi-stable and $\mu(\mathcal {M}_1)\leq\mu(\pi^*\Omega_{Y})<\mu(\Omega_{X})$ which  contradicts the property of the maximally destabilizing sub-sheaves $\mathcal {M}_1\subset \Omega_{X}$.

Now we have $\mathcal {M}_1\nsubseteq \pi^*\Omega_{Y}$.
By lemma \ref{subsheaves} $\rank\ I_B/I^2_B\cap\mathcal {N}_1^{\prime}|_B=\rank\mathcal {N}_1^{\prime}=1$ as $\mathcal {O}_{B}$-mod. So $\mathcal{G}\subset \Omega_{B}$ is a torsion sheaf as $\mathcal {O}_{B}$-mod and it is zero.
Therefore we have $$\mathcal {N}_1^{\prime}|_B\subseteq I_B/I^2_B$$
and
\begin{eqnarray*}
 \deg(\mathcal {N}_1^{\prime}|_B)\leq \deg (I_B/I^2_B)=\frac{-\deg B}{l}\label{ineq}
\end{eqnarray*}
and it is easy to check the following (in)equalities

\begin{eqnarray*}
\mu(\mathcal {M}_1)&=&\deg \mathcal {N}_1+(n-1)\deg B_1 \qquad  \text{(by the last line in diagram \ref{diag3.1} and lemma \ref{subsheaves})}\\
&=&n\deg \mathcal {N}_1^{\prime}+(n-1)\deg B
=n l\deg \mathcal {N}_1^{\prime}|_B+(n-1)\deg B\\
&\leq& -\deg B\hspace{30mm}  \text{(by the above inequality)}\\
&\leq&\frac{n\deg \Omega_{Y}+(n-1)\deg B}{2} \hspace{6mm}\text{(by the condition in the theorem)}   \\
&=&\mu(\Omega_{X})\hspace{32mm}(\text{by the second line in diagram \ref{diag3.1} and} \deg B=\deg B_1).
\end{eqnarray*}
We also get a contraction to the assumption. So $\Omega_{X}$ is  semi-stable with respect to $\pi^*\mathcal {H}.$
\end{proof}

\begin{thm}\label{thm3.6}
With the same conditions in Theorem \ref{thm3.5}, if the strict inequality holds: $n\deg \Omega_{Y}+(n+1) \deg B\ >0$ and the cotangent bundle $\Omega_{Y}$ is semi-stable with respect to $\mathcal {H}$, then $\Omega_{X}$ is stable with respect to $\pi^*\mathcal {H}\in \Pic(X)$.
\end{thm}
\begin{proof}
We know that $\Omega_X$ is semi-stable by Theorem \ref{thm3.5}.
We suppose that $\Omega_X$  is not stable with respect to $\pi^*\mathcal {H}$, then there is a Jordan-H\"{o}lder filtration:
$$0\subset\mathcal {M}_1\subset\Omega_X\ \text{ with }\ \mu(\mathcal {M}_1)=\mu(\Omega_X).$$
Now we get a saturated sub-line bundle $\mathcal {M}_1$. By the same argument as Theorem \ref{thm3.5}  we know that $\mathcal {M}_1^\sigma\neq\mathcal {M}_1$, then we get another saturated stable sub-line bundle $\mathcal {M}_1^\sigma$ with $\mu(\mathcal {M}_1^\sigma)=\mu(\Omega_X)$. Since the graded factors of Jordan-H\"{o}lder filtration are uniquely determined, the filtration splits: $$\Omega_X=\mathcal {M}_1\oplus \mathcal {M}_1^\sigma.$$

Now we want to prove that there is no such decomposition.
Since $\pi_*$ is exact we have the following diagram by pushing forward diagram \ref{diag3.1}:
$$\xymatrix{
0\ar[r] &{\pi_*\mathcal {N}_2} \ar[r] &{\pi_*\mathcal{M}_1^\sigma}\ar[r] &{\pi_*Q_2}\ar[r] &{0}
\\
0\ar[r] &{\pi_*\pi^*\Omega_{Y}} \ar[r]\ar@{->>}[u] &{\pi_*\Omega_X}\ar[r]\ar@{->>}[u] &{\pi_*\Omega_{X/Y}}\ar[r]\ar@{->>}[u] &{0} \\
0\ar[r] &{\pi_*\mathcal {N}_1} \ar[r]\ar@{^{(}->}[u] &{\pi_*\mathcal{M}_1}\ar[r]\ar@{^{(}->}[u] &{\pi_*Q_1}\ar[r]\ar@{^{(}->}[u] &{0.}
}
$$
where $\mathcal {N}_1=\pi^*\Omega_{Y}\cap\mathcal {M}_1$ and $\mathcal {N}_2=\pi^*\Omega_{Y}/\mathcal {N}_1$.
Note that
$$\pi_*\mathcal {N}_1\oplus \pi_*\mathcal {N}_1^{\sigma}\subseteq\pi_*\pi^*(\Omega_{Y})=\Omega_{Y}\oplus\Omega_{Y}(\mathcal {L}^{-1})\cdots\oplus\Omega_{Y}(\mathcal {L}^{\otimes(-n+1)}).$$
 where $\mathcal {N}_1^{\sigma}=\pi^*\Omega_{Y}\cap \mathcal {M}_1^\sigma$ and $\sigma$ acts on  $\Omega_{Y}(\mathcal {L}^{-i})$ by multiplying $\xi^{i}$, where $\xi$ is the $n$-th root of unity, and acts $\pi_*\mathcal {N}_1$ onto $\pi_*\mathcal {N}_1^{\sigma}$.
 Consider the sub-sheaf generated by $a+\sigma^*(a),\, a\in \pi_*\mathcal {N}_1$  whose rank is equal to $\rank \pi_*\mathcal {N}_1$,
 we have $\rank \Omega_{Y}=\rank \pi_*\mathcal {N}_1$ and $n=2$. So we only need to prove that such decomposition does not occur when $n=2$.

Note that $\pi_*\mathcal {N}_1\cap\Omega_{Y}=0,$
 and we  get an injection:
$$\pi_*\mathcal {N}_1\hookrightarrow\Omega_{Y}(\mathcal {L}^{-1})$$
with a nonzero torsion sheaf as its quotient when $n=2$. So we have
$$\mu(\pi_*\mathcal {N}_1)<\mu(\Omega_{Y}(\mathcal {L}^{-1}))=\mu(\Omega_{Y})-\deg (\mathcal {L}).$$

Now we relate some facts which will be used later. For any coherent sheaf $\mathcal {E}$ on $X$ we have
$$\text{deg}(\pi_*\mathcal {E})=\frac{\rank (\pi_*\mathcal {E})}{2}\text{deg}(\Omega_Y)-\frac{\rank (\mathcal {E})}{2}\text{deg}(\Omega_X)+\text{deg}(\mathcal {E}),$$ which can be calculated using Grothendieck-Riemann-Roch theorem. Note that $\rank(\pi_*\mathcal {E})=n\  \rank(\mathcal {E})$, we have
$$\mu(\pi_*\mathcal {E})=\mu(\Omega_Y)-\frac{1}{n}\mu(\Omega_X)+ \frac{1}{n}\mu(\mathcal {E})$$
if $\rank (\mathcal {E})>0$. By Lemma \ref{lem2.1} ii) we can get $\mu(\Omega_X)=\mu(\pi^*\Omega_Y)+\deg (\pi^*\mathcal {L}^{\otimes(n-1)})=n\mu(\Omega_Y)+n(n-1)\deg (\mathcal {L})$, so we have
$$\mu(\pi_*\mathcal {E})=\frac{1}{n}\mu(\mathcal {E})-(n-1)\deg (\mathcal {L})$$ if $\rank(\mathcal {E})>0$.

Next we will get a contradiction by calculating the slope of sheaves.

When $n=2$, by the above facts we have $$\mu(\pi_*\mathcal {N}_2)=\frac{1}{2}\mu(\mathcal {N}_2)-\deg (\mathcal {L})$$
and since the quotient $Q_2$ is torsion, $\mu(\mathcal {N}_2)<\mu(\mathcal {M}_1^{\sigma})=\mu(\Omega_X)$. So we have$$\mu(\pi_*\mathcal {N}_2)<\frac{1}{2}\mu(\Omega_X)-\deg (\mathcal {L})=\mu(\Omega_Y).$$


finally we conclude that
$$\mu(\pi_*\pi^*(\Omega_Y))=\frac{1}{2}(\mu(\pi_*\mathcal {N}_1)+\mu(\pi_*\mathcal {N}_2))<
\frac{1}{2}(\mu(\Omega_Y)-\text{deg}(\mathcal {L})+\mu(\Omega_{Y}))=\mu(\Omega_Y)-\frac{1}{2}\text{deg}(\mathcal {L}),$$
but $\mu(\pi_*\pi^*(\Omega_{Y}))=\mu(\Omega_{Y}\oplus\Omega_{Y}(\mathcal {L}^{-1}))=\mu(\Omega_Y)-\frac{1}{2}\text{deg}(\mathcal {L})$. So there is no such decomposition and $\Omega_X$ is stable  with respect to $\pi^*\mathcal {H}$.
\end{proof}

Next, we relate the following well known result.
\begin{thm}\label{thm3.7}(\cite[1.4]{HDLM})
Let $k$ be an algebraically closed field, and $\Omega^1_{\mathbb{P}^N}$ is the cotangent bundle on $\mathbb{P}^N_k$. Then $\Omega^1_{\mathbb{P}^N}$ is stable with respect to $\mathcal {O}_{\mathbb{P}^N}(1)$.
\end{thm}

Then we get the following result using Theorem \ref{thm3.7}.
\begin{cor}\label{cor3.8}
Let $k$ be an algebraically closed field with  $\Char k\nmid n$ and $\pi: X\rightarrow \mathbb{P}^2_k$ be a $n$-cyclic covering branched along a smooth curve $B$ in $\mathbb{P}^2_k$. Then $\Omega_X$ is semi-stable with respect to $\pi^*\mathcal {O}_{\mathbb{P}^N}(1)$, and it is stable with respect to $\pi^*\mathcal {O}_{\mathbb{P}^N}(1)$ when $n=2$ and $\deg B\geq 4$ or $n>2$.

\end{cor}
\begin{rmk}
In particular, when $n=2$ and $\deg B=6$ or $n=4$ and $\deg B=4$, $X$ is a K3 surface of degree $2$ or $4$ respectively. And the cotangent bundle $\Omega_X$ is stable with respect to $\pi^*\mathcal {O}_{\mathbb{P}^N}(1)$.
\end{rmk}
\section{Application}
\qquad Let $X$ be a smooth projective variety over an algebraically closed field $k$ with $\Char k=p>0$. The absolute Frobenius $F_X: X\rightarrow X$ is defined by the map: $f\mapsto f^p$ in $\mathcal {O}_X$. Let $F: X\rightarrow X_1:=X\times_kk$ denote the relative Frobenius morphism over $k$. We know that $F_*$ preserves semi-stability of vector bundles on curves of genus $g\geq1$.  When $W$ is a vector bundle on a higher dimensional variety, the instability of $F_*W$ is bounded by instability of $W\otimes {\rm T}^l(\Omega_X^1) (0\leq l\leq n(p-1))$ which is proved by X. Sun in \cite{S} and we will use this result to get our corollary.

\begin{defn}\label{defn4.1}(\cite[5.1]{J})
Let $V_0:=V=F^*(F_*W)$,
$V_1=\ker(F^*(F_*W)\surj W)$, and
$V_{l+1}:=\ker(V_{l}\xrightarrow{\nabla} V\otimes_{\mathcal {O}_X}
\Omega^1_X\to (V/V_{l})\otimes_{\mathcal {O}_X}\Omega^1_X)$
where
$\nabla: V\to V\otimes_{\mathcal {O}_X} \Omega^1_X$ is the canonical
connection with zero $p$-curvature(see \cite[Theorem 5.1]{K}). Then we have the filtration
$${0=V_{n(p-1)+1}\subset V_{n(p-1)}\subset\cdots\subset V_1\subset V_0=V=F^*(F_*W)}.$$
\end{defn}
Next, we recall the definition of ${\rm T}^l(V)$. Let $V$ be a n-dimensional vector space over $k$ with standard representation of ${\rm GL}(n)$, $S_l$ be the symmetric group with natural action on $V^{\otimes l}$ by
$(v_1\otimes\cdots\otimes
v_{l})\cdot\sigma=v_{\sigma(1)}\otimes\cdots\otimes
v_{\sigma(l)}$ for $v_i\in V$ and $\sigma\in{\rm S}_{l}$.
Let $e_1,\,\ldots,\,e_n$ be a basis of $V$, for $k_i\ge 0$ with
$k_1+\cdots+k_n=l$ then define
$${v(k_1,\ldots,k_n)=\sum_{\sigma\in{\rm S}_{l}}(e_1^{\otimes
k_1}\otimes\cdots\otimes e_n^{\otimes k_n})\cdot\sigma }.$$
Let ${\rm T}^{l}(V)\subset V^{\otimes l}$ be
the linear subspace generated by all vectors $v(k_1,\ldots,k_n)$
for all $k_i\ge 0$ satisfying $k_1+\cdots+k_n=l$.
Then we have the following property.

\begin{thm}\label{thm4.2}(\cite[Theorem 3.7]{S})
The filtration defined in Definition \ref{defn4.1} has the following properties
\begin{itemize}
\item[(i)]$\nabla(V_{l+1})\subset V_{l}\otimes\Omega^1_X$
for $l\ge 1$, and $V_0/V_1\cong W$.
\item[(ii)]
$V_{l}/V_{l+1}\xrightarrow{\nabla}(V_{l-1}/V_{l})\otimes\Omega^1_X$
are injective morphisms of vector bundles for $1\le l\le
n(p-1)$, which induced isomorphisms
$\nabla^{l}:
V_{l}/V_{l+1}\cong W\otimes_{\mathcal {O}_X}{\rm
T}^{l}(\Omega^1_X),\quad 0\le l\le n(p-1).$
\end{itemize}
\end{thm}
Let $\mathcal {H}$ be an ample line bundle on $X$. For any vector bundle $W$ on $X$, let
$${\rm I}(W,X):={\rm max}\{{\rm I}(W\otimes{\rm
T}^{l}(\Omega^1_X))\,|\,\, 0\le l\le n(p-1)\,\}$$ be the
maximal value of instabilities ${\rm I}(W\otimes{\rm
T}^{l}(\Omega^1_X))$.

\begin{thm}\label{thm4.3}(\cite[Theorem 4.8]{S}) When
$K_X\cdot{\mathcal {H}}^{n-1}\ge 0$, we have, for any $\sE\subset F_*W$,
$${\mu(F_*W)-\mu(\sE)\ge -\frac{{\rm I}(W,X)}{p}.}$$
In particular, if $W\otimes{\rm T}^{l}(\Omega^1_X)$, $0\le l\le
n(p-1)$, are semi-stable, then $F_*W$ is semi-stable. Moreover, if
$K_X\cdot{\mathcal {H}}^{n-1}>0$, the stability of the bundles
$W\otimes{\rm T}^{l}(\Omega^1_X)$, $0\le l\le n(p-1)$,
implies the stability of $F_*W$.
\end{thm}

When $X$ is a surface we have a better description of ${\rm T}^{l}(\Omega^1_X)$.
\begin{thm}\label{thm4.4}(\cite[Theorem 3.8]{S})
When ${\rm dim}(X)=2$, we have
$$V_{l}/V_{l+1}=\left\{
\begin{array}{llll} W\otimes{\rm Sym}^{l}(\Omega^1_X) &\mbox{when $l<p$}\\
W\otimes{\rm
Sym}^{2(p-1)-l}(\Omega^1_X)\otimes\omega_X^{l-(p-1)}&\mbox{when
$l\ge p$}
\end{array}\right.$$
\end{thm}

Now using the above results we get the following corollary.
\begin{cor}
Let $\pi:X\rightarrow\mathbb{P}^2$ be a $n$-cyclic covering over an algebraically closed field $k$ with $\Char k\nmid n$ branched along a smooth curve $B$ of degree $d$ and $\mathcal {H}=\pi^*\mathcal {O}(1)$. Then for any torsion free sheaf $\sL$ of rank $1$,
$F_*\sL$ is semi-stable if $d\geq5$ or $n=4$ and $d\geq4$. Moreover, $F_*W$ is stable if
 $n=2$ and $d\geq6$ or $n\geq3$ and $d\geq5$.
\end{cor}
\begin{proof}
First, by Lemma \ref{lem2.1} we know that $K_X\cdot{\mathcal {H}}\geq0$ if $d\geq5$ or $n=4$ and $d\geq4$ and $K_X\cdot{\mathcal {H}}> 0$ if $n=2$ and $d\geq6$ or $n\geq3$ and $d\geq5$.
And we have proved that $\Omega_X^1$ is stable in Corollary \ref{cor3.8}, so ${\rm Sym}^{l}(\Omega^1_X)$ is semi-stable when $l<p$. Then ${\rm T}^{l}(\Omega^1_X)$ is semi-stable by Theorem \ref{thm4.4} and ${\rm I}(\sL\otimes{\rm T}^{l}(\Omega^1_X))={\rm I}({\rm T}^{l}(\Omega^1_X))=0$ since $\sL$ is of rank $1$. At last by using Theorem \ref{thm4.3} we obtain our result.
\end{proof}

\end{document}